\documentclass[11pt]{amsart}
\pdfoutput=1

\usepackage{amsmath,amsthm,amssymb,mathrsfs,mathtools,stmaryrd,color}

\usepackage[top=1in, bottom=1in, left=1in, right=1in]{geometry}

\usepackage{float}
\usepackage{graphicx}
\usepackage{pdfpages}

\usepackage[utf8]{inputenc}
\usepackage[T1]{fontenc}

\usepackage{enumitem}
\setlist[enumerate]{itemsep=2pt,parsep=2pt,before={\parskip=2pt}}

\usepackage[colorlinks=true,hyperindex, linkcolor=red!60, pagebackref=false, citecolor=cyan, pdfpagelabels]{hyperref}
\usepackage[capitalize]{cleveref}

\usepackage{url}
\usepackage{breakurl}

\usepackage{tikz}
\usetikzlibrary{calc}
\usetikzlibrary{math}

\usepackage{thmtools}
\usepackage{thm-restate}

\newtheoremstyle{underline}
{}        
{}              
{}              
{\parindent}    
{}              
{}             
{1.5mm}         
{{\underline{\thmname{#1}\thmnumber{ #2}}~\thmnote{(#3)}}}

\newtheoremstyle{underlineBold}
{}        
{}              
{}              
{}    
{}              
{}             
{1.5mm}         
{{\underline{\textbf{\thmname{#1}\thmnumber{ #2}}}~\thmnote{(#3)}}}

\newtheoremstyle{underlineNoIndent}
{}        
{}              
{}              
{}    
{}              
{}             
{1.5mm}         
{{\underline{\thmname{#1}\thmnumber{ #2}.}}}

\newtheorem{theorem}{Theorem}[section]
\newtheorem*{theorem*}{Theorem}
\newtheorem*{definition*}{Definition}
\newtheorem{proposition}[theorem]{Proposition}
\newtheorem{lemma}[theorem]{Lemma}

\theoremstyle{definition}

\theoremstyle{underlineBold}
\newtheorem{case}{Case}
\theoremstyle{underline}
\newtheorem{subcase}{Case}[case]

\theoremstyle{underlineNoIndent}
\newtheorem{claim}{Claim}[subcase]

\begin{document}

\includepdf[pages=-]{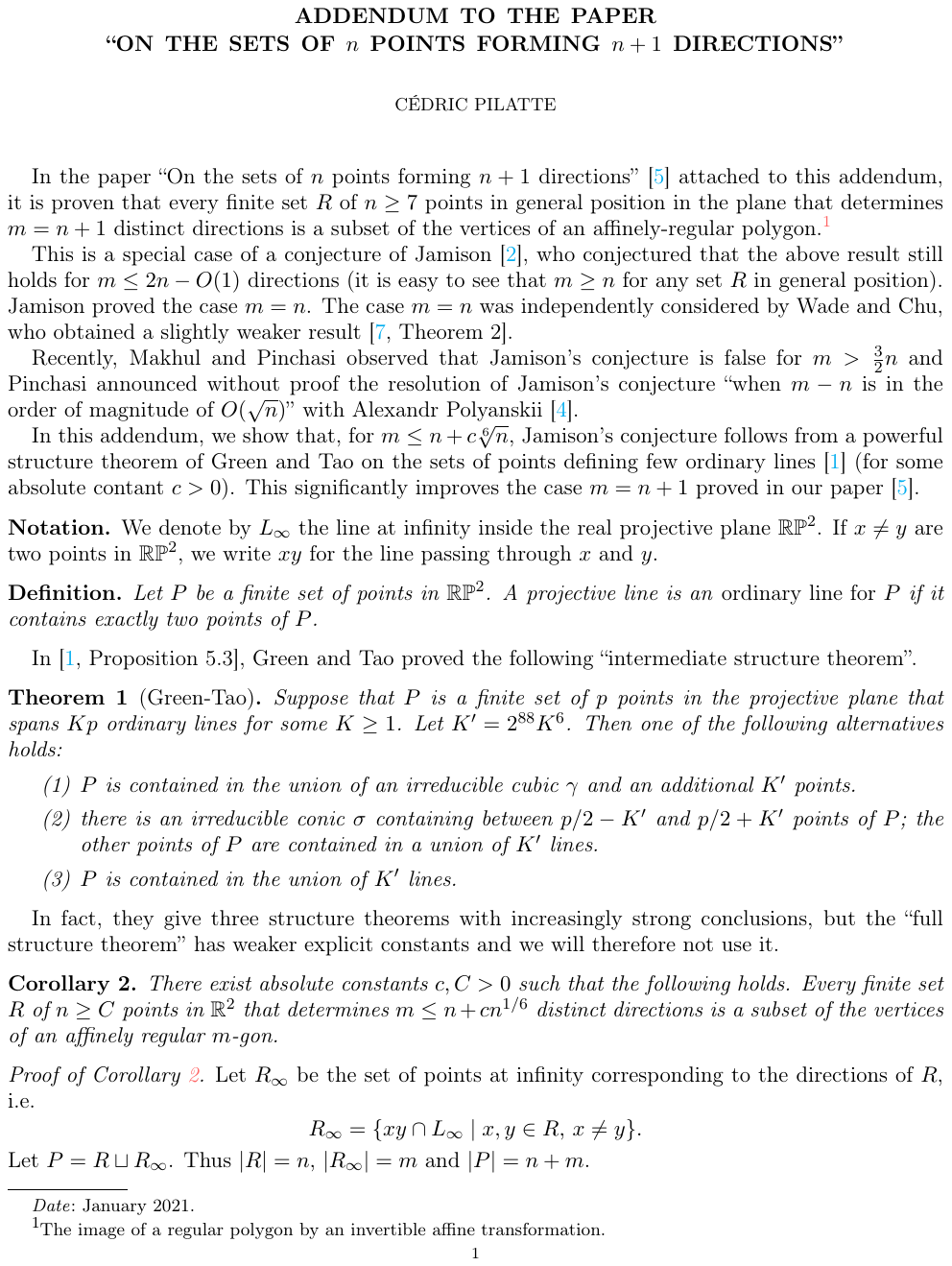}

\title{On the sets of $n$ points forming $n+1$ directions}

\author{C\'edric Pilatte}
\address{Department of Mathematics, University of Mons (UMONS), Place du Parc 20, 7000 Mons, Belgium.}
\date{October 2018}

\begin{abstract}
Let $S$ be a set of $n\geq 7$ points in the plane, no three of which are collinear. Suppose that $S$ determines $n+1$ directions. That is to say, the segments whose endpoints are in $S$ form $n+1$ distinct slopes. We prove that $S$ is, up to an affine transformation, equal to $n$ of the vertices of a regular $(n+1)$-gon. This result was conjectured in 1986 by R. E. Jamison.
\end{abstract}

\maketitle


\section{Introduction}

In 1970, inspired by a problem of Erd\H{o}s, Scott~\cite{Scott} asked the following question, now known as \emph{the slope problem}: what is the minimum number of directions determined by a set of $n$ points in $\mathbb{R}^2$, not all on the same line? By the number of directions (or slopes) of a set $S$, we mean the size of the quotient set $\{PQ\mid P,Q\in S, P\neq Q\}/\!\!\sim$, where $\sim$ is the equivalence relation given by parallelism: $P_1Q_1 \sim P_2Q_2\iff P_1Q_1\parallel P_2Q_2$.

Scott conjectured that $n$ points, not all collinear, determine at least $2\lfloor \frac n2\rfloor$ slopes. This bound can be achieved, for even $n$, by a regular $n$-gon; and for odd $n$, by a regular $(n-1)$-gon with its center. After some initial results of Burton and Purdy~\cite{BurtonPurdy}, this conjecture was proven by Ungar~\cite{Ungar} in 1982, using techniques of Goodman and Pollack~\cite{GoodmanPollack}. His beautiful proof is also exposed in the famous \emph{Proofs from the Book}~\cite[Chapter~11]{AignerZiegler}. Recently, Pach, Pinchasi and Sharir solved the tree-dimensional analogue of this problem, see\cite{Pinchasi2004, Pinchasi2007}.

A lot of work has been done to determine the configurations where equality in Ungar's theorem is achieved. A \emph{critical set} (respectively \emph{near-critical set}) is a set of $n$ non-collinear points forming $n-1$ slopes (respectively $n$ slopes). Jamison and Hill described four infinite families and 102 sporadic critical configurations~\cite{JamisonCritical1, JamisonCritical2, JamisonCatalogue}. It is conjectured that this classification is accurate for $n\geq 49$. No classification is known in the near-critical case. See~\cite{JamisonSurvey} for a survey of these questions, and other related ones.

In this paper, we suppose that no three points of $S$ are collinear (we say that $S$ is in general position). This situation was first investigated by Jamison~\cite{JamisonNoncollinear}, who proved that $S$ must determine at least $n$ slopes. As above, equality is possible with a regular $n$-gon. It is a well-known fact that affine transformations preserve parallelism. Therefore, the image of a regular $n$-gon under an affine transformation also determines exactly $n$ slopes.\footnote{A polygon obtained as the image of a regular polygon by an affine transformation is sometimes called an \emph{affine-regular} or \emph{affinely regular} polygon.} Jamison proved the converse, i.e. that the affinely regular polygons are the only configurations forming exactly $n$ slopes.

A much more general statement is believed to be true: for some constant $c_1$, if a set of $n$ points in general position forms $m\leq 2n-c_1$ slopes, then it is affinely equivalent to $n$ of the vertices of a regular $m$-gon (see~\cite{JamisonNoncollinear}). This would imply, in particular, that for every $c\geq 0$ and $n$ sufficiently large, every simple configuration of $n$ points determining $n+c$ slopes arises from an affinely regular $(n+c)$-gon, after deletion of $c$ points. Jamison's result thus shows it for $c=0$. Here, we will prove the case $c=1$. The general conjecture is still open. In fact, for $c\geq 2$, it is not even known whether the points of $S$ form a convex polygon. 

Every affinely regular polygon is inscribed in an ellipse. Conics will play an important role in our proof. Another problem of Elekes~\cite{Elekes} is the following: for all $m\geq 6$ and $C>0$, there exists some $n_0(m,C)$ such that every set $S\subset \mathbb{R}^2$ with $|S|\geq n_0(m,C)$  forming at most $C|S|$ slopes contains $m$ points on a (possibly degenerate) conic. It is still unsolved, even for $m=6$.

\section{Results}

\subsection{Preliminary Remarks}
\label{preliminaries}

Let $S$ be a set of $n$ points in the plane, in general position, that determines exactly $n+1$ slopes. If $S$ had a point lying strictly inside its convex hull, there would be at least $n+2$ slopes, as was proved by Jamison~\cite[Theorem~7]{JamisonNoncollinear}. Therefore, we know that we can label the points of $S$ as $A_1, \ldots ,A_n$, such that $A_1A_2\ldots A_n$ is a convex polygon. 

For every point $A_i\in S$, there are $n-1$ segments, with distinct slopes, joining $A_i$ to the other points of $S$. We will say that a slope is \emph{forbidden} at $A_i$ if it is not the slope of any segment $A_iA_j$, for $j\neq i$. Since $S$ determines $n+1$ slopes, \emph{there are exactly two forbidden slopes} at each point of $S$.

We will denote by $\nabla A_iA_j$ the slope of the line $A_iA_j$. Thus, an equality like $\nabla A_{i_1}A_{i_2}=\nabla A_{i_3}A_{i_4}$ is equivalent to $A_{i_1}A_{i_2}\parallel A_{i_3}A_{i_4}$.
Throughout our main proof, we will repeatedly make use of the next lemma. It will be particularly useful to prove that a slope is forbidden at a point or that two slopes are equal. As an obvious corollary, we have that $\nabla A_{i-1}A_{i+1}$ is forbidden at $A_i$ for all $i\in \mathbb Z$. Throughout the paper, when we say ``for all $i\in\mathbb Z$'', we consider the indices modulo $n$, so that $A_{n+1}:=A_1$, and so on.

\begin{lemma}
Let $1 \leq i < j < k \leq n$. Exactly one of the following is true:
\begin{itemize}
  \item the slope of $A_iA_k$ is forbidden at $A_j$;
  \item $\exists p,i<p<k$ such that $A_iA_k \parallel A_jA_p$.
\end{itemize}
Moreover, in the second case, $\nabla A_jA_p\notin \{ \nabla A_iA_l \mid l\neq j, k\}\cup \{\nabla A_lA_j\mid l\neq i,k\}$.
\label{lem:forbidden}
\end{lemma}

\begin{proof}
This is almost immediate from the definition of a forbidden slope. In the second case, if $p\in \{1,\ldots , n\}$ were not between $i$ and $k$, the segments $A_iA_k$ and $A_jA_p$ would intersect. Finally, if $\nabla A_jA_p$ were equal to some $\nabla A_iA_l$, then $A_iA_k\parallel A_jA_p \parallel A_iA_l$, so $A_i, A_k$ and $A_l$ would be aligned, a contradiction. The same is true for the segments $A_lA_j$.
\end{proof}

We will also need the following result, which can be found in~\cite[Chapter~1]{PrasolovSolovyev}.

\begin{proposition}
Let $\mathcal C$ be a non-degenerate conic and $O$ a point on $\mathcal C$. If $P,Q$ are two points on $\mathcal C$, define $P+Q$ to be the unique point $R$ on $\mathcal C$ such that $RO\parallel PQ$ (with the convention that $XX$ is the tangent to $\mathcal C$ at $X$, for $X\in \mathcal C$). This addition turns $\mathcal C$ into an abelian group, of which $O$ is the identity element.
\label{prop:group}
\end{proposition}

In particular, for $P,Q,R,S$ four points on $\mathcal C$, we have $P+Q=R+S$ if and only if $PQ\parallel RS$. Five points in general position determine a unique conic. \Cref{lem:pascal} will enable us to introduce conics in the proof, in order to use \cref{prop:group}.

\begin{lemma}
Suppose $P_1,\ldots, P_6$ are points in the plane such that $P_1P_6\parallel P_2P_5$, $P_2P_3\parallel P_1P_4$ and $P_4P_5\parallel P_3P_6$. Then $P_1, \ldots , P_6$ lie on a common conic. \label{lem:pascal}
\end{lemma}

\begin{proof}
This follows immediately from Pascal's theorem, with the hexagon $\mathcal H =P_1P_4P_5P_2P_3P_6$. Indeed, the intersections of the opposite sides of $\mathcal H$ are collinear on the line at infinity.
\end{proof}

\begin{figure}[H]
\centering
\begin{tikzpicture}[scale=0.6]
\draw [rotate around={141.69624887687556:(3.4782608695652173,-0.8260869565217391)}] (3.4782608695652173,-0.8260869565217391) ellipse (5.389241694860877cm and 2.623805418319392cm);

\node (1) at (-1, 2) {};
\node (6) at (0, -1) {};
\node (2) at (0, 3) {};
\node (5) at (2, -3) {};
\node (4) at (6.25, 0.1875) {};
\node (3) at (4, 2) {};

\node [left] at (-1, 2) {$P_1$};
\node [left] at (0, -1) {$P_6$};
\node [above] at (0, 3) {$P_2$};
\node [left] at (2, -3) {$P_5$};
\node [right] at (6.25, 0.1875) {$P_4$};
\node [right] at (4, 2) {$P_3$};
	
\draw [fill=black] (-1,2) circle (2.5pt);
\draw [fill=black] (0,-1) circle (2.5pt);
\draw [fill=black] (0,3) circle (2.5pt);
\draw [fill=black] (2,-3) circle (2.5pt);
\draw [fill=black] (6.25,0.1875) circle (2.5pt);
\draw [fill=black] (4,2) circle (2.5pt);

\draw [thick, dashed] (1.center) -- (6.center);
\draw [thick, dashed] (2.center) -- (5.center);
\draw [thick, dotted] (2.center) -- (3.center);
\draw [thick, dotted] (1.center) -- (4.center);
\draw [thick, dash dot dot] (3.center) -- (6.center);
\draw [thick, dash dot dot] (4.center) -- (5.center);

\end{tikzpicture}
\caption{Illustration of \cref{lem:pascal}.}
\label{fig:conic}
\end{figure}

For the reader's convenience, we reproduce here a result of Korchm\'aros~\cite{Korchmaros} (which is also discussed in~\cite{FisherJamison}), that we will use twice in the proof.

\begin{lemma}
Let $P_1,\ldots , P_n$ be distinct points on a non-degenerate conic. Suppose that, for all $j\in \mathbb Z$, $P_{j+1}P_{j+2}\parallel P_{j}P_{j+3}$. Then, $P$ is affinely equivalent to a regular $n$-gon.
\label{lem:polygon}
\end{lemma}

\subsection{Main Theorem}

\begin{theorem}
Any set $S$ of $n\geq 7$ points in the plane, in general position, that determines exactly $n+1$ slopes, is affinely equivalent to $n$ of the vertices of a regular $(n+1)$-gon.
\end{theorem}

\begin{proof}

We use the notations of \cref{preliminaries}: $S=\{A_1,A_2,\ldots, A_n\}$ where $A_1A_2\ldots A_n$ is a convex polygon. We will split the proof into two cases. In the first case, we suppose that, for every $i\in \mathbb Z$, $A_{i+1}A_{i+2}\parallel A_iA_{i+3}$. If this fails for some $i$, we can assume that this $i$ is $1$.

\begin{case}For every $i\in \mathbb Z$, $A_{i+1}A_{i+2}\parallel A_iA_{i+3}$. \smallbreak

We will distinguish subcases according to which segments are parallel to $A_iA_{i+5}$. As we will see, none of the subcases is actually possible. 

\begin{subcase}For all $i \in \mathbb Z$, $A_iA_{i+5}\parallel A_{i+1}A_{i+4}$. \smallbreak
Let $A_{k+1}, \ldots , A_{k+6}$ be any six consecutive points of $S$. We have $A_{k+1}A_{k+6}\parallel A_{k+2}A_{k+5}$, $A_{k+2}A_{k+3}\parallel A_{k+1}A_{k+4}$ and $A_{k+4}A_{k+5}\parallel A_{k+3}A_{k+6}$ from our two assumptions. Thus, \cref{lem:pascal} implies that the six points lie on a common conic. As this is true for any six consecutive points, and since five points in general position (i.e. no three collinear) determine a unique conic, all the $A_i$'s lie on the same conic. Together with the fact that $\forall i, \,A_{i+1}A_{i+2}\parallel A_iA_{i+3}$, this implies that $A_1A_2\ldots A_n$ is affinely equivalent to a regular $n$-gon, by \cref{lem:polygon}. Therefore, $S$ determines exactly $n$ directions, which is a contradiction.
\end{subcase}

\begin{subcase}For some $i\in \mathbb Z$, we have $A_{i}A_{i+5}\parallel A_{i+2}A_{i+4}$. \label{case:1.2}\smallbreak
Say $i=1$, meaning $A_{1}A_{6}\parallel A_3A_{5}$. By \cref{lem:forbidden} applied three times, we see that $\nabla A_2A_{6}$ is forbidden at $A_{3}, A_{4}$ and $A_{5}$ (here, we have used that $A_{2}A_{6}\nparallel A_3A_{5}$ and, for $A_{4}$, that $A_{3}A_{4}\parallel A_{2}A_{5}$ and $A_{4}A_{5}\parallel A_{3}A_{6}$). For $l=3,4,5$, we know that $\nabla A_{2}A_6$ and $\nabla A_{l-1}A_{l+1}$ are exactly the two forbidden slopes at $A_l$. Therefore, $\nabla A_1A_5$ is not forbidden at $A_4$, hence, by \cref{lem:forbidden} again, we conclude that $A_1A_5\parallel A_2A_4$. Similarly, $\nabla A_3A_7$ is not forbidden at $A_4$, so $A_3A_7\parallel A_4A_6$. As the slope of $A_2A_7$ is not forbidden at $A_5$, we conclude $A_2A_7\parallel A_4A_5(\parallel A_3A_6)$. We have $A_3A_4\parallel A_2A_5$, $A_3A_6\parallel A_2A_7$ and we just showed that $A_2A_7\parallel A_3A_6$. By \cref{lem:pascal}, $A_2, A_3, \ldots, A_7$ lie on a common conic.

\begin{figure}[ht]
\centering
\begin{tikzpicture}[scale=0.5, rotate=-40]
\node (1) at (9.12,-6.67) {};
\node (2) at (8.29,-3.24) {};
\node (3) at (5.22,-0.63) {};
\node (4) at (1.67,-0.34) {};
\node (5) at (-0.16,-2.54) {};
\node (6) at (-3.48,-11.13) {};
\node (7) at (-6.26,-20.84) {};

\draw [fill=black] (1) circle (1pt);
\draw [fill=black] (2) circle (1pt);
\draw [fill=black] (3) circle (1pt);
\draw [fill=black] (4) circle (1pt);
\draw [fill=black] (5) circle (1pt);
\draw [fill=black] (6) circle (1pt);
\draw [fill=black] (7) circle (1pt);

\node [right] at (1) {$A_1$};
\node [right] at (2) {$A_2$};
\node [above] at (3) {$A_3$};
\node [above] at (4) {$A_4$};
\node [above] at (5) {$A_5$};
\node [left] at (6) {$A_{6}$};
\node [right] at (7) {$A_{7}$};

\draw [] (3.center) -- (4.center);
\draw [] (2.center) -- (5.center);
\draw [double] (1.center) -- (6.center);
\draw [double] (3.center) -- (5.center);
\draw [dashed] (1.center) -- (4.center);
\draw [dashed] (2.center) -- (3.center);
\draw [dotted] (5.center) -- (6.center);
\draw [dotted] (7.center) -- (4.center);
\draw [dash dot dot] (5.center) -- (4.center);
\draw [dash dot dot] (3.center) -- (6.center);

\end{tikzpicture}
\caption{Case \ref{case:1.2}.}
\end{figure}

We will equip this conic with the group structure descibed in \cref{prop:group}, with $A_7$ the zero element. We will write $A_7=0$ and $A_6=x$. Then, $A_5A_6\parallel A_4A_7$, $A_4A_5\parallel A_3A_6$ and $A_4A_6\parallel A_3A_7$ together imply $A_5=2x$, $A_4=3x$ and $A_3=4x$. Also, $A_3A_4\parallel A_2A_5$ gives $A_2=5x$. Let $B$ be the point on the conic with $B=6x$. We thus have $A_2A_3\parallel BA_4$ and $A_2A_4\parallel BA_5$. However, there can only be one point $P$ with $A_2A_3\parallel PA_4$ and $A_2A_4\parallel PA_5$. As $A_1$ is such a point, $A_1=B=6x$. This contradicts $A_1A_6\parallel A_3A_5$, as $A_1+A_6=6x+x\neq 4x+2x=A_3+A_5$.
\end{subcase}

\begin{subcase}For some $i\in \mathbb Z$, we have $A_{i}A_{i+5}\parallel A_{i+1}A_{i+3}$.\smallbreak
This is exactly the previous case after having relabelled every $A_i$ as $A_{n+1-i}$.
\end{subcase} 

\begin{subcase} The previous cases do not apply. \label{case:1.4}\smallbreak
If none of the previous cases is possible, there must be some $i$, say $i=1$, for which $A_1A_{6}$ is not parallel to any of $A_{2}A_{5}$, $A_{3}A_{5}$ and $A_{2}A_{4}$. Then, $\nabla A_1A_6$ is forbidden at $A_2, A_3, A_4$ and $A_5$. Once again, we deduce that the forbidden slopes at $A_l$, $2\leq l\leq 5$, are $\nabla A_1A_6$ and $\nabla A_{l-1}A_{l+1}$. We use \cref{lem:forbidden} to find $A_2A_6\parallel A_3A_5$ (applied with $A_k=A_4$) and  $A_1A_5\parallel A_2A_4$ ($A_k=A_2$).

Let $\mathcal C$ be the conic passing through $A_1, A_2, \ldots , A_5$. We use \cref{prop:group} to define a group structure on $\mathcal C$, with $A_1=0$. Let $A_2=x$ and $A_3=y$. From $A_2A_3\parallel A_1A_4$ and $A_3A_4\parallel A_2A_5$, we have $A_4=x+y$ and $A_5=2y$. But $A_2A_4\parallel A_1A_5$ implies $y=2x$, so $A_i=(i-1)x$ for $1\leq i \leq 5$. We use the same argument as before. Let $B=5x$, then $A_4A_5\parallel A_3B$ and $A_3A_5\parallel A_2B$, so $B=A_6=5x$. We deduce $A_1A_6\parallel A_2A_5$, a contradiction.
\end{subcase}

\begin{figure}[ht]
\centering
\begin{tikzpicture}[scale=0.9, rotate=15]
\node (5) at (-0.41, -0.73) {};
\node (3) at (1,1) {};
\node (2) at (3,1) {};
\node (1) at (5.34,-0.73) {};
\node (0) at (5.95,-3.25) {};
\node (-1) at (4.05, -3.4) {};

\draw [fill=black] (5) circle (1pt);
\draw [fill=black] (3) circle (1pt);
\draw [fill=black] (2) circle (1pt);
\draw [fill=black] (1) circle (1pt);
\draw [fill=black] (0) circle (1pt);
\draw [fill=black] (-1) circle (1pt);

\node [left] at (5) {$A_6$};
\node [above] at (3) {$A_5$};
\node [above] at (2) {$A_4$};
\node [right] at (1) {$A_3$};
\node [right] at (0) {$A_2$};
\node [right] at (-1) {$A_{1}$};

\draw [dash dot dot] (1.center) -- (0.center);
\draw [dash dot dot] (2.center) -- (-1.center);
\draw [] (1.center) -- (2.center);
\draw [] (3.center) -- (0.center);
\draw [double] (-1.center) -- (5.center);
\draw [dashed] (1.center) -- (5.center);
\draw [dashed] (2.center) -- (3.center);

\draw [densely dotted] (1.center) -- (3.center);
\draw [loosely dotted] (0.center) -- (2.center);
\end{tikzpicture}
\caption{Case \ref{case:1.4}.}
\end{figure}

\end{case}

\begin{case}We have $A_2A_3\nparallel A_1A_4$ (without loss of generality).\smallbreak

Without loss of generality, we can also suppose that the point $A_4$ is closer to the line $A_2A_3$ than is $A_1$. In this situation, the line parallel to $A_2A_3$ passing through $A_4$ intersects the segment $[A_1A_2]$ in its relative interior, and the line parallel to $A_2A_3$ passing through $A_1$ does not intersect the segment $[A_3A_4]$.

From $A_2A_3\nparallel A_1A_4$, we deduce that the forbidden slopes at $A_2$ and $A_3$ are $\nabla A_1A_3,\nabla A_1A_4$ and $\nabla A_2A_4$, $\nabla A_1A_4$, respectively. Thus, $A_1A_2\parallel A_nA_3$ and $A_2A_5\parallel A_3A_4$. We now show that $A_2A_3$ is forbidden at $A_4$. Suppose, for some $k$, that $A_2A_3\parallel A_4A_k$. Then, $k$ has to be between $5$ and $n$, so $A_1A_2A_3A_4A_k$ must be a convex polygon, with $A_2A_3\parallel A_4A_k$. We can see that this contradicts the fact that $A_4$ is closer than $A_1$ to the line $A_2A_3$.

\begin{subcase}$A_{n-1}A_2\parallel A_nA_1$.\label{case:2.1} \smallbreak
We want to show that this case is impossible. From \cref{lem:forbidden}, we find $A_{n-1}A_3\parallel A_{n}A_{2}$. When we apply this lemma again with the slope of $A_nA_4$, we find that $A_nA_4$ is parallel to $A_1A_3$, because $A_2A_3$ is forbidden at $A_4$. In the same way, we get $A_{n-1}A_4\parallel A_nA_3$.

\begin{figure}[ht]
\centering
\begin{tikzpicture}[scale=1, rotate=30]
\node (5) at (0, -1) {};
\node (4) at (0, 0) {};
\node (3) at (1.5,1) {};
\node (2) at (3,1) {};
\node (1) at (5,-1) {};
\node (0) at (6,-3.5) {};
\node (-1) at ($(3,1)+2.4*(1,-2.5)$) {};

\draw [fill=black] (6,-3.5) circle (1pt);
\draw [fill=black] (5,-1) circle (1pt);
\draw [fill=black] (3,1) circle (1pt);
\draw [fill=black] (1.5,1) circle (1pt);
\draw [fill=black] (0,0) circle (1pt);
\draw [fill=black] (0, -1) circle (1pt);
\draw [fill=black] (-1) circle (1pt);

\node [left] at (5) {$A_5$};
\node [left] at (4) {$A_4$};
\node [above] at (3) {$A_3$};
\node [above] at (2) {$A_2$};
\node [right] at (1) {$A_1$};
\node [right] at (0) {$A_n$};
\node [right] at (-1) {$A_{n-1}$};

\draw [dash dot dot] (1.center) -- (0.center);
\draw [dash dot dot] (2.center) -- (-1.center);
\draw [dashed] (1.center) -- (2.center);
\draw [dashed] (3.center) -- (0.center);
\draw [dotted] (2.center) -- (5.center);
\draw [dotted] (3.center) -- (4.center);
\draw [] (2.center) -- (3.center);
\draw [double] (1.center) -- (4.center);
\end{tikzpicture}
\caption{Case \ref{case:2.1}.}
\end{figure}

Let $\mathcal C$ be the conic passing through $A_3, A_2, A_1, A_n$ and $A_{n-1}$. Again, we use \cref{prop:group}, setting $A_{n-1}=0$. Let $A_n=x$ and $A_2=y$. From $A_{n-1}A_3\parallel A_nA_2$ we deduce $A_3=x+y$, and from $A_{n-1}A_4\parallel A_nA_3$ we get $A_4=y+2x$. Let $B=2x$. Then $A_nA_4\parallel BA_3$ and $A_nA_3\parallel BA_2$. This means that $B$ belongs to the line parallel to $A_nA_4$ through $A_3$ and to the line parallel to $A_nA_3$ through $A_2$. So $B=A_1$, i.e. $A_1=2x$. On the one hand, the relation $A_{n-1}A_2\parallel A_nA_1$ gives $0+y=x+2x$. On the other hand, $A_{2}A_3\nparallel A_1A_4$ yields $y+(y+x)\neq 2x+(y+2x)$. This is absurd.
\end{subcase}

\begin{subcase}$A_{n-1}A_2\nparallel A_nA_1$\label{case:2.2}. \smallbreak
This is the last case of the proof, and the only case that produces valid configurations of points. As $A_{n-1}A_2\nparallel A_nA_1$, $\nabla A_{n-1}A_2$ is forbidden at $A_1$. With $\nabla A_0A_2$, those are the two forbidden slopes at $A_1$. Therefore, none of $\nabla A_2A_i$, $3\leq i\leq n-2$ is forbidden at $A_1$. So, every $\nabla A_2A_i$, $3\leq i\leq n-2$, corresponds to a unique $\nabla A_1A_j$ for some $j$.

\begin{figure}[H]
\centering
\begin{tikzpicture}[scale = 3.8, rotate=-30, y=0.7cm]
\node (0) at (1.0,0.0) {};
\fill (1.0,0.0) circle (0.3pt);
\node (1) at (0.9135454576426009,0.40673664307580015) {};
\fill (0.9135454576426009,0.40673664307580015) circle (0.3pt);
\node (2) at (0.6691306063588583,0.7431448254773941) {};
\fill (0.6691306063588583,0.7431448254773941) circle (0.3pt);
\node (3) at (0.30901699437494745,0.9510565162951535) {};
\fill (0.30901699437494745,0.9510565162951535) circle (0.3pt);
\node (5) at (-0.4999999999999998,0.8660254037844388) {};
\fill (-0.4999999999999998,0.8660254037844388) circle (0.3pt);
\node (6) at (-0.8090169943749473,0.5877852522924732) {};
\fill (-0.8090169943749473,0.5877852522924732) circle (0.3pt);
\node (7) at (-0.9781476007338057,0.2079116908177593) {};
\fill (-0.9781476007338057,0.2079116908177593) circle (0.3pt);
\node (8) at (-0.9781476007338057,-0.20791169081775907) {};
\fill (-0.9781476007338057,-0.20791169081775907) circle (0.3pt);
\node (9) at (-0.8090169943749475,-0.587785252292473) {};
\fill (-0.8090169943749475,-0.587785252292473) circle (0.3pt);
\node (10) at (-0.5000000000000004,-0.8660254037844384) {};
\fill (-0.5000000000000004,-0.8660254037844384) circle (0.3pt);
\node (11) at (-0.10452846326765423,-0.9945218953682733) {};
\fill (-0.10452846326765423,-0.9945218953682733) circle (0.3pt);
\node (12) at (0.30901699437494723,-0.9510565162951536) {};
\fill (0.30901699437494723,-0.9510565162951536) circle (0.3pt);
\node (13) at (0.6691306063588585,-0.743144825477394) {};
\fill (0.6691306063588585,-0.743144825477394) circle (0.3pt);
\node (14) at (0.913545457642601,-0.40673664307580015) {};
\fill (0.913545457642601,-0.40673664307580015) circle (0.3pt);

\node [right] at (3) {$A_1$};
\node [above] at (5) {$A_2$};
\node [above] at (6) {$A_3$};
\node [left] at (7) {$A_4$};
\node [left] at (8) {$A_5$};
\node [right] at (2) {$A_{n}$};
\node [right] at (1) {$A_{n-1}$};
\node [right] at (0) {$A_{n-2}$};

\node [left] at (10) {$A_{i-2}$};
\node [left] at (12) {$A_i$};

\draw (6.center) -- (5.center);
\draw [double] (7.center) -- (3.center);
\draw [dashed] (3.center) -- (5.center);
\draw [dashed] (6.center) -- (2.center);
\draw [dotted] (6.center) -- (7.center);
\draw [dotted] (5.center) -- (8.center);
\draw [dash dot dot] (3.center) -- (2.center);
\draw [double, dash dot dot] (5.center) -- (1.center);

\draw (5.center) -- (10.center);
\draw (3.center) -- (12.center);
\end{tikzpicture}
\caption{Case \ref{case:2.2}.}
\end{figure}

A simple but important observation is that, for all $3\leq i_1,i_2\leq n-2$ and $4\leq j_1,j_2\leq n$, 
$$\begin{cases}A_2A_{i_1} \parallel A_1A_{j_1}\\
A_2A_{i_2} \parallel A_1A_{j_2}
\end{cases}\Longrightarrow \quad \left( i_1<i_2\iff j_1<j_2\right).$$
That is, the assignment $f$ that maps every $3\leq i\leq n-2$ to the unique $4\leq j\leq n$ such that $A_2A_i\parallel A_1A_j$ must be strictly increasing. Moreover, it has to satisfy $f(3)\neq 4$ as we assumed $A_2A_3\nparallel A_1A_4$. The unique possibility is then $f(i)=i+2$ for every $i$. We have proven that, for $5\leq i\leq n$, $A_2A_{i-2} \parallel A_1A_{i}$.

\begin{claim}For every $i\in \{5,\ldots, n\}$,
\begin{enumerate}
  \item $A_{i-2}A_i$ and $A_2A_{i-2}$ are the two forbidden slopes at $A_{i-1}$, and;
  \item $\forall k\in \{3,\ldots , i-2\}, \, A_{i-1}A_k\parallel A_iA_{k-1}$.
\end{enumerate}
\end{claim}

\begin{proof}[Proof of claim]
For $i=5$, we have already proven those two statements. We will prove them for $i=j$, assuming it has already been proven for all $5\leq i\leq j-1$.
\begin{enumerate}
  \item We have to show that $A_2A_{j-2}$ is forbidden at $A_{j-1}$. This is clear as $A_2A_{j-2}\parallel A_1A_j$ and there is no point of $S$ between $A_1$ and $A_2$.
  \item Since we know the forbidden slopes at $A_{j-1}$, we can use \cref{lem:forbidden} at the point $A_{j-1}$ several times, with different slopes. First, $\nabla A_{j}A_{j-3}$ is not forbidden, so $A_jA_{j-3}\parallel A_{j-1}A_{j-2}$. Then $\nabla A_{j}A_{j-4}$ is not forbidden, and is distinct from $\nabla A_{j-1}A_{j-2}=\nabla A_jA_{j-3}$, so $A_{j}A_{j-4}\parallel A_{j-1}A_{j-3}$. We can continue this way, until we get $A_{j}A_{2}\parallel A_{j-1}A_{3}$. This proves the claim. \qedhere
\end{enumerate}
\end{proof}

In particular, for every $i\in\{6,\ldots ,n-1\}$, we have $A_3A_i\parallel A_{2}A_{i+1}$, $A_5A_{i}\parallel A_4A_{i+1}$. As $A_3A_4\parallel A_2A_5$, we can use \cref{lem:forbidden}, which shows that $A_2,A_3,A_4,A_5,A_i$ and $A_{i+1}$ lie on a conic. As this is true for every $6\leq i \leq n-1$, we know that the $A_i$'s, for $2\leq i\leq n$, all lie on a common conic (because there is a unique conic passing through five points in general position). 

As we have done several times in this proof, we use the group structure on the conic given by parallelism. Choose $A_2$ to be the identity element, let $A_3=x$. Solving 
$$\begin{cases}
  A_3A_4\parallel A_2A_5\\
  A_3A_6\parallel A_4A_5\\
  A_3A_5\parallel A_2A_6
\end{cases}$$
gives $A_4=2x$, $A_5=3x$ and $A_6=4x$. Then, a simple induction (using $A_{i-1}A_3\parallel A_{i}A_2$) gives $A_i=(i-2)x$ for all $i\in\{2,\ldots ,n\}$. Let $B$ be the point on the conic with $B=-2x$. Then $A_2A_3\parallel BA_5$ and $A_2A_4\parallel BA_6$. However, we proved before that $A_2A_3\parallel A_1A_5$ and $A_2A_4\parallel A_1A_6$, so $A_1=B=-2x$.

To summarize, we know that all the $n$ points of $S$ are on a conic, $A_i=(i-2)x$ for $i\in \{2,\ldots , n\}$ and $A_1=-2x$. We use the group structure one last time: $A_3A_n\parallel A_1A_2$ implies $x+(n-2)x=-2x+0$, so $(n+1)x=0$. Therefore, the subgroup generated by $A_3=x$ is a finite cyclic group of order $n+1$: 
$$\langle A_3\rangle=\Big\{\underbrace{A_{2}}_{0}, \underbrace{A_{3}}_{x}, \underbrace{A_{4}}_{2x},\ldots,\underbrace{A_{n-1}}_{(n-3)x}, \underbrace{A_{n}}_{(n-2)x}, \underbrace{A_{1}}_{(n-1)x}, -x\Big\}.$$
To finish the proof, we use the more convenient notations $P_j:=jx$ for $0\leq j\leq n$ (so that every $A_i$ is a $P_j$). If the indices are considered modulo $n+1$, we have, for all $j\in \mathbb Z$, $P_{j+1}P_{j+2}\parallel P_{j}P_{j+3}$, because $(j+1)x+(j+2)x=jx+(j+3)x$. By \cref{lem:polygon}, $P_0P_1P_2\ldots P_n$ is, up to an affine transformation, a regular $(n+1)$-gon. \qedhere
\end{subcase}

\end{case}

\end{proof}

\section{Acknowledgements}

I would like to thank Christian Michaux, without whom this work would not have been possible.


\end{document}